\documentclass{amsart} 
\usepackage{enumerate, amssymb, amsfonts, latexsym, esint}
\usepackage{amscd, amsmath}
\usepackage{graphicx}
\newtheorem{theorem}{Theorem} 
\newtheorem{lemma}{Lemma}

\newtheorem{corollary}{Corollary}

\newtheorem{claim}{Claim}
\newtheorem{remark}{Remark}
\begin{document} 

\title{The Fully Nonlinear Stochastic Impulse Control Problem}

\author{Rohit Jain}

\begin{abstract}
In this paper, motivated by a problem in stochastic impulse control theory, we aim to study solutions to a free boundary problem of obstacle-type. We obtain sharp estimates for the solution using nonlinear tools which are independent of the modulus of semi-convexity of the obstacle. This allows us to state a general estimate for solutions to free boundary problems of obstacle-type admitting obstacles with a general modulus of semi-convexity. We provide two applications of our result. We consider penalized fully nonlinear obstacle problems and provide sharp decay estimates for H\"{o}lder norms and we prove sharp estimates for the solution to a fully nonlinear stochastic impulse control problem.
\end{abstract}

\maketitle

\section{Introduction}

Stochastic impulse control problems (\cite{BL82}, \cite{L73}, \cite{M76}, \cite{F79}) are control problems that fall between classical diffusion control and optimal stopping problems. In such problems the controller is allowed to instantaneously move the state process by a certain amount every time the state exits the non-intervention region. This allows for the controlled process to have sample paths with jumps. There is an enormous literature studying stochastic impulse control models and many of these models have found a wide range of applications in electrical engineering, mechanical engineering, quantum engineering, robotics, image processing, and mathematical finance. A key operator in stochastic impulse control problems is the intervention operator 
\\
\begin{equation}
\textnormal{M}u(x) = \inf_{\xi \geq 0} (u(x+\xi) + 1).
\end{equation}
\\
The operator represents the value of the strategy that consists of taking the best immediate action in state $x$ and behaving optimally afterward. Since it is not always optimal to intervene, this leads to the quasi-variational inequality 
\\
\begin{equation}
u(x) \leq \textnormal{M}u(x) \; \; \forall x \in \mathbb{R}^{n}.
\end{equation}
\\
From the analytic perspective one obtains an obstacle problem where the obstacle depends implicitly and nonlocally on the solution. More precisely we can consider the classical stochastic impulse control problem
\\
\[
   \begin{cases}
         \Delta u(x) \geq f(x)& \forall x \in \Omega.\\
         u(x) \leq Mu(x) & \forall x \in \Omega.\\
         u = 0 & \forall x \in \partial \Omega.        
   \end{cases}
\]
\\
Here we let $\Omega \subset \mathbb{R}^{n}$ be a bounded domain with a $C^{2,\alpha}$ boundary $\partial \Omega$, $u \in W_{0}^{1,2}(\Omega)$, $f \in L^{\infty}(\Omega)$, $ f \geq 0$, and 
\begin{equation}
\textnormal{M}u(x) = \inf_{\stackrel{\xi \geq 0}{x + \xi \in \bar{\Omega}}}(u(x + \xi) + 1).
\end{equation} 
The assumption $f \geq 0$ implies that the solution $\bar{u}$ to the boundary value problem $\Delta \bar{u} = f$ in $\Omega$ with $\bar{u} \in H_{0}^{1}(\Omega)$ satisfies $ \bar{u} \geq 0$. This implies in particular that the set of solutions to $v \leq M\bar{u}$ with $v \in H_{0}^{1}(\Omega)$ is nonempty. This allows for an iterative procedure to prove existence and uniqueness of the solution (\cite{BL82}). We also point out that the sharp $C^{1,1}_{loc}$ estimate in the classical stochastic impulse control problem has been previously obtained (\cite{CF79a}, \cite{CF79b}). In this paper we consider a fully nonlinear problem. We let $F(D^{2}u)$ be a fully nonlinear uniformly elliptic operator i.e. $\lambda\|P\| \leq F(A+P) - F(P) \leq \Lambda\|P\| \; \forall P \geq 0$ and $A,P \in \mathbf{S}(n)$, where $\mathbf{S}(n)$ is the set of all $n \times n$ real symmetric matrices. We also assume that the operator is either convex or concave in the hessian variable. We define $\varphi_{u}(x)$ to be a semi-convex function with a general modulus of semi-convexity $\omega(r)$. We consider the following boundary value problem. 
\\
\begin{equation}
   \begin{cases}
        F(D^{2}u) \leq 0 & \forall x \in \Omega. \\
        u(x) \geq \varphi_{u}(x) & \forall x \in \Omega. \\    
        u = 0 & \forall x \in \partial \Omega. \\ 
  \end{cases}
\end{equation}
\\
In this work we are interested in proving sharp estimates for fully nonlinear obstacle problems admitting obstacles with a general modulus of semi-convexity. The following are our main results in this paper, 

\begin{theorem}  Consider the fully nonlinear obstacle problem with obstacle $\varphi_{u}$, admitting a modulus of semiconvexity, $\omega(r)$. Then the solution $u$ has a modulus of continuity $\omega(r)$ up to $C^{1,1}(\Omega)$.
\end{theorem}

As an application we apply our result to obtain a sharp estimate for the solution to the following fully nonlinear stochastic impulse control problems,

\begin{theorem} Let $\Omega \subset \mathbb{R}^{n}$ be a bounded domain with a $C^{2,\alpha}$ boundary $\partial \Omega$. Define
\\
\begin{equation} 
\textnormal{M}u(x) = \varphi(x) + \inf_{\stackrel{\xi \geq 0}{x + \xi \in \bar{\Omega}}}(u(x + \xi)).
\end{equation}
\\
Here $\varphi(x)$ is $\omega(r)$ semiconcave, strictly positive, bounded, and decreasing in the positive cone $\xi \geq 0$. Consider the solution to the following fully nonlinear stochastic impulse control problem,
\\
\begin{equation}
   \begin{cases}
        F(D^{2}u) \geq 0 & \forall x \in \Omega.\\
        u(x) \leq \textnormal{M}u(x) & \forall x \in \Omega.\\
        u = 0 & \forall x \in \partial \Omega.        
  \end{cases}
\end{equation}
\\ 
Then, the solution $u$ has modulus of continuity $\omega(r)$ up to $C^{1,1}(\Omega)$.
\end{theorem}

We remark that as a corollary of this theorem we recover the sharp $C^{1,1}$ estimate for the classical stochastic impulse control problem. 
\\
\\
We proceed in stages to prove the stated theorems. The main point of interest in the first theorem is to improve the modulus of continuity for the obstacle $\varphi_{u}$ on the contact set $\{u = \varphi_{u}\}$. In particular the goal is to obtain a uniform modulus of continuity $\omega(r)$ for $\varphi_{u}$ which we can then propogate to the solution $u$. The second theorem follows from the first theorem once we establish semiconcavity estimates for the nonlocal obstacle $\textnormal{M}u(x)$. Moreover we can extend the free boundary regularity from the classical implicit constraint problem as considered in \cite{J15C} under the assumption that the data is analytic and $\varphi(x) = 1$. Finally as an application of the previous results we consider a singularly perturbed fully nonlinear obstacle problem and show optimal decay rates for H\"{o}lder norm estimates. 
\\
\\
\emph{Acknowledgements} I would like to express my sincerest gratitude and deepest apprecation to my thesis advisors Professor Luis A. Caffarelli and Professor Alessio Figalli. It has been a truly rewarding experience learning from them and having their guidance. 

\section{Lipschitz Estimates for the Solution}

To obtain the optimal estimate we first prove initial regularity estimates which we hope to extend. We fix $f =0$. All the proofs may be modified for a nonzero sufficiently regular $f$. We begin by first proving that solutions are indeed continuous. 

\begin{lemma} Let $u$ be a solution to (4) with semiconvex obstacle $\varphi_{u}$. Then $u \in C(\Omega)$. 
\end{lemma}

The lemma follows from a result due to G.C. Evans.

\begin{lemma} If $u$ is continuous in $\{ u = \varphi_{u} \}$, then $u$ is continuous in $\Omega$.
\end{lemma}

\begin{proof}
The possibility of a discontinuity is limited to a point on the free boundary, $\partial \{u > \varphi_{u} \}$. Consider $x_{0} \in \{u = \varphi_{u} \}$ and without loss of generality assume $u(x_{0}) = 0$. Suppose by contradiction, that there exists a sequence of points $\{x_{k}\}$ with the following properties:
\\
1. $\{x_{k}\} \to x_{0}$. 
\\
2. $\forall k$, $x_{k} \in \{u > \varphi_{u} \}.$
\\
3. $\mu = \lim_{x_{k} \to x_{0}} u(x_{k}) > u(x_{0}) = 0.$
\\
\\
By lower semicontinuity of $u$, we know that $\forall \delta > 0$ there exists a neighborhood of $x_{0}$ such that $u \geq -\delta$ for $\delta << \mu$. We consider
$$r_{k} = \textnormal{dist} [ x_{k}, \{u = \varphi_{u} \} ].$$
For a large enough $k$ we can ensure that:
\\
1. $u(x) + \delta \geq 0 \; \; \textnormal{in} \; B_{r_{k}}(x_{k}).$
\\
2. $u(x_{k}) + \delta \geq \frac{\mu}{2}.$
\\ 
\\
Moreover we know that $u(x) + \delta$ satisfies the equation in $B_{r_{k}}(x_{k})$. By the Harnack Inequality we obtain,
$$\frac{\mu}{2} \leq u(x_{k}) + \delta \leq C \inf_{B_{\frac{r_{k}}{2}(x_{k})}} (u + \delta).$$
This implies for some $C_{0} > 0$ universal,
$$ \inf_{B_{\frac{r_{k}}{2}(x_{k})}} u \geq C_{0} \mu.$$ 

Since $u$ is also superharmonic, we know from the weak Harnack Inequality in $B_{4r_{k}}(y_{k})$ that,
\[
\begin{split}
u(y_{k}) & \geq c \left( \fint_{B_{2r_{k}}(y_{k})} u^{p} \right)^{\frac{1}{p}} \\
             & = \frac{c}{|B_{2r_{k}}|^{1/p}} \left( \int_{B_{2r_{k}}(y_{k}) \setminus B_{\frac{r_{k}}{2}}(x_{k})} u^{p} + \int_{B_{\frac{r_{k}}{2}}(x_{k})} u^{p} \right)^{\frac{1}{p}} \\
             & \geq \frac{c}{|B_{2r_{k}}|^{1/p}} \left(-(\delta)^{p} |B_{2r_{k}}| + (C_{0} \mu)^{p} |B_{\frac{r_{k}}{2}}| \right)^{\frac{1}{p}} \\
             & \geq C_{1} \mu \; \; \; \textnormal{for} \; C_{1} > 0.
\end{split}
\]

On the other hand $u(y_{k}) = \varphi_{u}(y_{k})$ and $y_{k} \to x_{0}$. This implies in particular that,
\\
1. $\varphi_{u}(y_{k}) \geq C_{1} \mu.$
\\
2. $\varphi(x_{0}) = u(x_{0}) = 0.$
\\
This is our desired contradiction. 
\end{proof} 

\begin{remark} We observe that the conditions on the obstacle may be relaxed in the proof of this lemma. In fact continuity of the obstacle is sufficient.
\end{remark}

A generic semiconvex function with a general modulus of semiconvexity is known to be Lipschitz in the interior. We extend the previous result to show that solutions to a fully nonlinear obstacle problem admitting obstacles with a Lipschitz modulus of continuity grow from the free boundary with a comparable rate. 

\begin{lemma} Let $u$ be a solution to (4) with semiconvex obstacle $\varphi_{u}$. Fix $0 \in \partial\{u > \varphi_{u}\}.$ Then 
$$\sup_{B_{r}(0)} u(x) \leq Cr.$$
\end{lemma}

\begin{proof}
Let $\gamma(r)$ denote the Lipschitz modulus of continuity for the obstacle $\varphi_{u}$ in $B_{r}(0)$. The obstacle condition $u \geq \varphi_{u}$ implies in $B_{r}(0)$ that 
$$u \geq \varphi_{u}(0) - \gamma(r).$$

Define
$$v(x) = u - (\varphi_{u}(0) - \gamma(r)).$$

We note that $F(D^{2}v) = F(D^{2}u) \leq 0$ and $F(D^{2}v) = 0$ inside $\{u > \varphi_{u}\}.$ We consider $x \in B_{r/4}(0) \cap \{u > \varphi_{u}\}.$ Moreover we let $y$ be the closest free boundary point to $x$. Let $\rho$ be the distance of $x$ to its closest free boundary point $y$. From the Weak Harnack Inequality it follows,
$$v(y) \geq C \left(\fint_{B_{2\rho(y)}} v^{p} \right)^{1/p}.$$

By the positivity of $v$ and Harnack Inequality in $B_{\rho(x)}$,  it follows that the right hand side,
$$\geq C \left( \frac{B_{\rho}(x)}{B_{2 \rho(y)}} \fint_{B_{\rho(x)}} v^{p} \right)^{1/p} \geq Cv(x). $$

Recall that $|\varphi_{u}(0) - \varphi_{u}(y)| \leq \gamma(r).$ Hence,
$$0 \leq v(y) = \varphi_{u}(y) - \varphi_{u}(0) + \gamma(r) \leq 2\gamma(r).$$

Changing back to our solution $u$, we find,
$$0 \leq u(x) - (u(0) - \gamma(r)) \leq Cv(y) \leq C \gamma(r).$$

In particular,
$$u(x) - u(0) \leq C \gamma(r).$$
\end{proof}

\section{Optimal $C^{\omega(r)}$ Estimates for the Solution}

In the previous section we assumed that the obstacle had a uniform modulus of continutity. A priori for semi-concave functions you only know that the uniform modulus of continuity is Lipschitz. Our goal as in the classical case will be to study the interplay between the equation and the obstacle to improve regularity estimates for the obstacle on the contact set. We start this section by stating and proving a lemma in the particular case that our operator is the Laplacian. The motivating calculation will help us proceed to prove the desired estimate in the more general case. The content of the lemma says that for any given point $x_{1} \in \{u(x) > \varphi_{u}(x) \}$,  $\exists x_{0} \in \{u(x) =  \varphi_{u}(x)\}$ such that the solution grows at most by $\omega(2|x_{1} - x_{0}|)$ where $\omega(|x_{1} - x_{0}|)$ denotes the modulus of semiconvexity for the obstacle on the ball $B_{|x_{1} - x_{0}|}(x_{0})$. Since a lower estimate is available via the obstacle, what we aim to show is that around a fixed contact point the modulus of continuity of the solution is controlled by the modulus of semiconvexity of the obstacle.

\begin{lemma} Let  $\varphi_{u}(x)$ be a semiconvex function with general modulus of semiconvexity $\omega(r)$. Consider the following obstacle problem:
\begin{equation}
   \begin{cases}
        \Delta u \leq 0 & \forall x \in \Omega.\\
        u(x) \geq \varphi_{u}(x) & \forall x \in \Omega.\\
        u = 0 & \forall x \in \partial \Omega      
  \end{cases}
\end{equation}

Fix  $x\in \{u(x) > \varphi_{u}(x) \}$ and define $L_{x_{0}}(x) = \varphi_{u}(x_{0}) + \langle p, x-x_{0} \rangle$, the linear part of the obstacle at the point $x_{0}$.  Then $\exists x_{0} \in \{u(x) = \varphi_{u}(x)\}$ and $C(n) >0$  such that $u(x) - L_{x_{0}}(x) \leq C(n)\omega(2|x - x_{0}|)$. 
\end{lemma}    

\begin{proof}
We fix  $x_{1} \in \{u(x) > \varphi_{u}(x) \}$. Let $x_{0}$ denote the closest point to $x_{1}$ in $\{u  = \varphi_{u} \}$. We denote this distance by $\rho = |x_{1} - x_{0}|$. Define $w(x) = u(x) - L_{x_{0}}(x)$. Using the mean value theorem for superharmonic functions in $B_{2\rho}(x_{0})$ we have,
\[
\begin{split}
 0 & = w(x_{0}) \geq \frac{1}{\alpha(n)2^{n}\rho^{n}} \int_{B_{2\rho}(x_{0})} w(y) \;dy\\
 & = K(n)\int_{B_{2\rho} (x_{0}) \smallsetminus B_{\rho} (x_{1})} w(y) \; dy \; + K(n) \int_{B_{\rho} (x_{1})} w(y) \; dy.
\end{split}
\]

Semiconvexity of $w(x)$ in $B_{2\rho}(x_{0})$ and an application of the mean value theorem for harmonic functions in $B_{\rho}(x_{1})$ implies,
\[
\begin{split}
 & \geq K(n) \int_{B_{2\rho} (x_{0}) \smallsetminus B_{\rho} (x_{1})} - \omega(2|y-x_{0}|) + C_{1}(n) w(x_{1})\\
 & \geq -\tilde{C}(n)\omega(2\rho) + C_{1}(n) w(x_{1}).
\end{split}
\]

In particular we obtain the desired bound,
$$w(x_{1}) \leq C(n) \omega(2\rho).$$
\end{proof} 

We now look to generalize the previous argument in the fully nonlinear setting. In the preceding proof the lower bound on the obstacle was transferred to the solution at the contact point. Moreover we were able to renormalize the solution by subtracting off a linear part. We will also need a generalization of the mean value theorem that was used to connect pointwise information with information about the measure $\Delta u$. 
\\
\\
We consider again (4). For clarity we set $\omega(r) = \bar{C}r^{2}$ for some positive constant $\bar{C} > 0$. The arguments presented below can be trivially modified for the general semiconvex modulus by an appropriate rescaling. We make a remark in this direction towards the end of this section.

\begin{lemma} Let $x_{1} \in \{u > \varphi_{u} \}$. Then $\exists x_{0} \in \{u = \varphi_{u} \}$ such that for $w(x) = u(x) - L_{x_{0}}(x)$, where $L_{x_{0}}(x) = \varphi_{u}(x_{0}) + \langle p, x-x_{0} \rangle$ denotes the linear part of the obstacle at the point $x_{0}$, and a universal constant $K(n) > 0$, 
$$w(x_{1}) \leq K(n) |x_{1} - x_{0}|^{2}.$$
\end{lemma}

\begin{proof}
Fix $x_{1} \in \{u > \varphi_{u} \}$. Let $x_{0}$ be the closest point to $x_{1}$ in $\{u = \varphi_{u} \}$.  We denote this distance by $\rho = |x_{1} - x_{0}|$. By the modulus of semiconvexity of the obstacle we know that $w(x) \geq -16\bar{C}\rho^{2}$ on $B_{4\rho}(x_{0})$. The idea of the proof is to zoom out to scale 1 and prove that the solution is bounded by a universal constant and then rescale back to obtain the desired bound. Consider the transformation $y = \frac{x-x_{0}}{\rho}$ and the scaled solution,
\begin{equation} 
v(y) = \frac{w(\rho y + x_{0})}{16\bar{C}\rho^{2}} + 1.
\end{equation}

We note that $v(y)$ is a non-negative supersolution on $B_{4}(0)$ with 
$$\inf_{B_{4}(0)} v(y) \leq 1.$$ 

Moreover $v(y)$ is a solution in $B_{1}(y_{1})$, since $x_{0}$ is the closest point in the contact set to $x_{1}$. By the interior Harnack Inequality,
$$v(y_{1}) \leq \sup_{B_{\frac{1}{2}}(y_{1})} v(y) \leq C \inf_{B_{\frac{1}{2}}(y_{1})} v(y).$$ 

We also know from the weak $L^{\epsilon}$ estimate for supersolutions that for universal constants $d$, $\epsilon$, 
$$|\{v \geq t\} \cap B_{2}(0)| \leq dt^{-\epsilon} \; \; \forall t >0.$$ 

We observe that $B_{\frac{1}{2}}(y_{1}) \subseteq B_{2}(0)$. Hence we can choose $t = t_{0}$ such that 
$$t_{0} = \left(\frac{\delta d}{|B_{\frac{1}{2}(y_{1})}|} \right)^{\frac{1}{\epsilon}},$$
for $\delta > 0$. It follows that,
$$|\{v \leq t\} \cap B_{\frac{1}{2}}(y_{1})| > \delta |B_{\frac{1}{2}}(y_{1})| > 0.$$

Hence there exists a universal constant $C$ such that,
$$v(y_{1}) \leq C.$$ 

This implies from (8), 
$$ \frac{w(\rho y_{1} + x_{0})}{4\bar{C}\rho^{2}} \leq C.$$ 

Rescaling back we find,
$$w(x_{1}) \leq K|x_{1} - x_{0}|^{2}.$$
\end{proof} 

We now prove a lemma that controls the oscillation of the solution between two arbitrary points on the contact set $\{u = \varphi_{u} \}$. As a corollary which we state after the proof, we improve the modulus of continuity for the obstacle $\varphi_{u}$ on the contact set $\{u = \varphi_{u}\}$.

\begin{lemma}Let $x_{1} \in \{u = \varphi_{u} \}$ and  $x_{0} \in \{u = \varphi_{u} \}$. Then for $w(x) = u(x) - L_{x_{0}}(x)$, where $L_{x_{0}}(x) = \varphi_{u}(x_{0}) + \langle p, x-x_{0} \rangle$ denotes the linear part of the obstacle at the point $x_{0}$, and $K(n) > 0$ a universal constant, 
$$w(x_{1}) \leq K(n) |x_{1} - x_{0}|^{2}.$$
\end{lemma} 

\begin{proof}
Assume by contradiction that for an arbitrary large constant $K > 0$
\begin{equation}
w > K|x_{1} - x_{0}|^{2}. 
\end{equation}

As before we denote the distance between the points by $\rho = |x_{1} - x_{0}|$. We begin with a claim.

\begin{claim} $\exists \;  \textnormal{Half Ball} \; HB_{\rho}(x_{1}) \; \textnormal{such that} \; \forall x \in HB_{\rho}(x_{1}), \; w(x) \geq \frac{K}{2} \rho^{2}.$
\end{claim}

\begin{proof} 
We define $\varphi_{w} = \varphi_{u} - L_{x_{0}}$, where as before $L_{x_{0}} = \varphi_{u}(x_{0}) + \langle p, x-x_{0} \rangle$ for $p \in D^{+} \varphi_{u}(x_{0})$ the superdifferential of $\varphi_{u}$ at the point $x_{0}$. We make the following observations:
\\
1. $w \geq \varphi_{w} \; \; \; \forall x \in B_{2\rho}(x_{0}).$
\\
2. $\varphi_{w}(x_{1}) = \varphi_{u}(x_{1}) - \varphi_{u}(x_{0}) + \langle p, x-x_{1} \rangle - \langle p, x-x_{0} \rangle \; \; \; \forall x \in B_{2\rho}(x_{0}).$
\\
\\
In particular,
$$w(x) \geq \varphi_{w}(x_{1}) + \varphi_{u}(x) - \varphi_{u}(x_{1}) - \langle p, x - x_{1} \rangle.$$

Now consider $d \in D^{+} \varphi_{u}(x_{1})$ and observe that $w(x_{1}) = \varphi_{w}(x_{1})$. This produces the following inequality, 
$$w(x) \geq w(x_{1}) + \varphi_{u}(x) - \varphi_{u}(x_{1}) - \langle d, x - x_{1} \rangle - \langle p-d, x - x_{1} \rangle.$$

By semiconvexity on $B_{\rho}(x_{1})$, (9), and fixing $x \in HB_{\rho}(x_{1}) = \{ x \in B_{\rho}(x_{1}) \; | \; \langle p-d, x - x_{1} \rangle \leq 0 \},$ we have,
$$w(x) \geq K\rho^{2} - \bar{C} \rho^{2}.$$

We can choose $K$ large enough so that we obtain,
$$ w(x) \geq \frac{K}{2}\rho^{2}.$$

This is our desired half ball.
\end{proof}

We now consider again the dilated solution $v(y)$ from the previous lemma (8). By the weak Harnack Inequality for supersolutions, $\exists C > 0$ universal and $\epsilon > 0$ such that,
$$\int_{B_{4}(0)} |v(x)|^{\epsilon/2} \leq \left(C \inf_{B_{2}(0)}v(x)\right)^{\frac{2}{\epsilon}} \leq \left(Cv(0) \right)^{\frac{2}{\epsilon}} = C.$$

From our previous claim we obtain 
$$ 0 < |\{v(y) > \frac{K}{32\bar{C}} \} \cap B_{1}(y_{1})|.$$

Here $\bar{C}$ is our semiconvexity constant from before. We now have the following chain of inequalities,
\[
\begin{split}
 0 & < |\{v(y) > \frac{K}{32\bar{C}} \} \cap B_{1}(y_{1})|\; (\frac{K}{32\bar{C}})^{\epsilon /2} \\
 & = \int_{\{v(x) > \frac{K}{32\bar{C}} \} \cap B_{1}(y_{1})}  (\frac{K}{32\bar{C}})^{\epsilon/2} \\
 & \leq \int_{\{v(x) > \frac{K}{32\bar{C}} \} \cap B_{1}(y_{1})}   |v(x)|^{\epsilon /2} \\
 & \leq \int_{B_{4}(0)} |v(x)|^{\epsilon/2} \leq \left(C \inf_{B_{2}(0)}v(x) \right)^{\frac{2}{\epsilon}} \leq \left(Cv(0) \right)^{\frac{2}{\epsilon}} = C.
\end{split}
\]

For $K$ large enough we obtain a contradiction. Hence for a universal constant $K(n) > 0$, 
$$w(x_{1}) \leq K(n) |x_{1} - x_{0}|^{2}.$$ 
\end{proof}

\begin{remark}
Assume our obstacle is semiconvex on $B_{r}(x)$ with modulus of semiconvexity $\omega(r)$. We can translate our solution to the origin and scale by the modulus of semiconvexity of the obstacle. In particular, set $\rho$ to be the distance between our fixed points. 
\begin{equation}
v(y) =\frac{w(\rho y + x_{0})}{\omega(4\rho)}+1. 
\end{equation}

One can check that we get similar estimates in terms of the modulus of semiconvexity $\omega(\rho)$. 
\end{remark}

\begin{remark} A corollary of the previous lemma is that on the contact set $\{u = \varphi_{u} \}$ the obstacle $\varphi_{u}$ has a modulus of continuity $\omega(r)$. In particular,
\begin{equation} 
\|\varphi_{u}\|_{C_{loc}^{\omega(r)}(\{u = \varphi_{u} \})} \leq C.
\end{equation}
\end{remark}

We can now state and prove a sharp estimate for our solutions.

\begin{theorem}
Consider the boundary value problem (4) with semiconvex obstacle $\varphi_{u}$ admitting a modulus of semiconvexity, $\omega(r)$. Then the solution $u$ has modulus of continuity $\omega(r)$ up to $C^{1,1}(\Omega)$. In particuluar,
\begin{equation}
\|u\|_{C^{\omega(r)}(\Omega)} \leq C.
\end{equation}
\end{theorem}

\begin{proof}
To prove this theorem we consider three distinct cases.
\\
\textbf{Case 1}: $x_{1} \in \{ u > \varphi_{u}\}$, $x_{0} \in \{ u = \varphi_{u}\}$. 
\\
Choose the closest point in the contact set to $x_{1}$ and call it $\bar{x}_{1}$. Then we apply Lemma 5 to obtain the correct oscillation estimate up to the free boundary. Then an application of Lemma 6 gives us the correct oscillation estimate between two contact points. Finally we use the triangle inequality to conclude. 
\\
\\
\textbf{Case 2}: $x_{1}$, $x_{0}$ $\in \{ u = \varphi_{u}\}$. 
\\
This is the content of Lemma 6.
\\
\\
\textbf{Case 3}: $x_{1}$, $x_{0}$ $\in \{ u > \varphi_{u}\}$. 
\\
We distinguish two different subcases.
\\
\emph{Case 3a}: $\max\{d(x_{1} , \{ u = \varphi_{u}\}), d(x_{0} , \{ u = \varphi_{u}\}) \} \geq 4|x_{1} - x_{2}|$. 
\\
Suppose $\max\{d(x_{1} , \{ u = \varphi_{u}\}), d(x_{2} , \{ u = \varphi_{u}\}) \} = \rho$. Without loss of generality we assume that the maximum distance is realized at the point $x_{1}$. We observe that $B_{|x_{1}-x_{0}|}(x_{1}) \subseteq B_{\frac{\rho}{2}}(x_{1})$, and we consider $w = u - L_{x_{1}}$, where $L_{x_{1}}$ denotes the linear part of the solution at $x_{1}$. By an application of the Harnack Inequality we obtain,
$$ \sup_{B_{\rho}(x_{1})} w \leq C \inf_{B_{\rho / 2}(x_{1})} w \leq Cw(x_{2}) \leq C\omega(\rho).$$ 

Moreover we also appeal to the interior estimates for solutions to our fully nonlinear convex or concave operator, $F(D^{2}u) = 0$, 
$$\|w - w(x_{1}) \|_{C^{\omega(\rho)}(B_{\frac{\rho}{2}}(x_{1}))} \leq \frac{K}{\omega(\rho)}\|w - w(x_{1}) \|_{L^{\infty}(B_{\rho}(x_{1}))}.$$  

Hence, 
$$\|w - w(x_{1}) \|_{C^{\omega(\rho)}(B_{\frac{\rho}{2}}(x_{1}))} \leq C.$$ 

\emph{Case3b}: $\max\{d(x_{1} , \{ u = \varphi_{u}\}), d(x_{2} , \{ u = \varphi_{u}\}) \} < 4|x_{1}-x_{2}|$ 
\\
In this case one considers $\rho_{1} = d(x_{1} , \{ u = \varphi_{u}\})$ and $\rho_{0} = d(x_{0} , \{ u = \varphi_{u}\})$. Let $\bar{x}_{1}$ be the closest contact point to $x_{1}$ and $\bar{x}_{0}$ the closest contact point to $x_{0}$. We can apply Lemma 5 to obtain the desired oscillation estimate for each point up to the free boundary. We then apply Lemma 6 to control the oscillation between two contact points. Finally we apply the triangle inequality to conclude.   
\end{proof}

\section{Application to Stochastic Impulse Control Theory}

In the previous section we obtained a general estimate for fully nonlinear obstacle problems admitting an obstacle with a general modulus of semi-convexity. In this section we would like to apply the estimate to a particular fully nonlinear obstacle problem arising in stochastic impulse control theory. The idea is to prove that the given obstacle $Mu(x)$ is semi-concave with modulus of semiconcavity $ \omega(r)$. The strategy of the proof will follow the ideas presented in (\cite{CF79a}). We point out that the existence and uniqueness of a continuous viscosity solution to the fully nonlinear stochastic impulse control problem follows from introducing the Pucci Extremal Operators (see chapter 2 in \cite{CC95}) and adapting the arguments in (\cite{I95}). What one also needs to do is use Evans Lemma iteratively on a sequence of solutions to the fully nonlinear obstacle problem with continuous obstacle (See Remark 1).  

\begin{theorem} Let  $\varphi(x)$ be $\omega(r)$ semi-concave, strictly positive, bounded, and decreasing in the positive cone $\xi \geq 0$. Then the Obstacle 
$$Mu(x) = \varphi(x) + \inf_{\stackrel{\xi \geq 0}{x + \xi \in \Omega}}u(x + \xi)$$ 
is semi-concave with modulus of semi-concavity $\omega(r)$.
\end{theorem}

\begin{proof}
We consider two distinct cases: 
\\
1. $x_{0} \in \{u = Mu\}.$
\\
2. $x_{0} \in \{u < Mu\}.$
\\
\\
\textbf{Case 1}: Fix $x_{0} \in \{u = Mu \}.$ 
\\
The proof in this case is based on characterizing the set where the infimum of $u$ occurs and establishing that this set is uniformly contained in the non-contact region $\{u < Mu \}$. This is the content of the following claims. We define the following sets: 
\\
1. $\Sigma_{\geq x_{0}} = \{x_{0} + \xi \; : \; \xi \geq 0\}.$
\\
2. $\Sigma_{x_{0}} = \{\varphi(x_{0}) + u(x_{0} + \xi) = Mu(x_{0})\}.$
\\
\\
The following claim characterizes $\Sigma_{x_{0}}$ as the set of points where $u$ realizes its infimum.
\begin{claim} For every $y \in (\Sigma_{\geq x_{0}} \setminus \Sigma_{x_{0}})$ and for every $x \in \Sigma_{x_{0}}$, $u(x) \leq u(y).$
\end{claim}

\begin{proof}
Fix $\bar{x} \in \Sigma_{x_{0}}$. Suppose by contradiction that $\exists x_{1} \in \Sigma_{\geq x_{0}} \setminus \Sigma_{x_{0}}$ such that $u(x_{1}) < u(\bar{x})$. This implies the following chain of inequalities,
\\
\[
\begin{split}
 \varphi(x_{0}) + u(x_{1}) &  < \varphi(x_{0}) + u(\bar{x})\\
 & = Mu(x_{0}) = \varphi(x_{0}) + \inf_{\stackrel{\xi \geq 0}{x_{0} + \xi \in \Omega}}u(x_{0} + \xi).
\end{split}
\]
\\
In particular we obtain,
$$u(x_{1}) < \inf_{\stackrel{\xi \geq 0}{x_{0} + \xi \in \Omega}}u(x_{0} + \xi).$$
This is our desired contradiction. 
\end{proof}

We now prove that pointwise the elements of $\Sigma_{x_{0}}$ are contained in the non-contact region, $\{u < Mu\}$.
\begin{claim} Suppose the solution to the Boundary Value Problem $F(D^{2}\bar{u}) = 0$ satisfies 
$$\bar{u} < \inf_{\partial \Omega} \varphi.$$ 
Then $\forall x \in \Sigma_{x_{0}}$ it follows that $u(x) < Mu(x)$. 
\\
Moreover in a neighborhood $N_{1}$ of $x$ we have $u \in C^{1,1}(N_{1})$
\end{claim}

\begin{proof}
We observe that the first statement ensures that $\Sigma_{x_{0}} \cap (\partial \Omega) = \varnothing$. Suppose $x_{0} \in \Omega^{\circ}$, $x \in \partial \Omega$ and $x_{0} \leq x$. Then we observe,
\\
\[
\begin{split}
 Mu(x_{0}) & = u(x_{0})  \\
 & \leq \bar{u}(x_{0}) < \inf_{\partial \Omega} \varphi \leq \varphi(x) + u(x)  \leq \varphi(x_{0}) + u(x).
\end{split}
\]
\\
The last inequality follows because $\varphi(x)$ is monotonically decreasing in the cone. Hence in particular  $\Sigma_{x_{0}} \cap (\partial \Omega) = \varnothing$. 
\\
\\
Suppose now by contradiction that $\exists x \in \Sigma_{x_{0}}$ such that $u(x) = Mu(x).$ Then we have the following chain of inequalities,
\\
\[
\begin{split}
 u(x_{0}) & = Mu(x_{0}) \\
 & = \varphi(x_{0}) + u(x)\\
 & = \varphi(x_{0}) +  Mu(x) \geq \varphi(x_{0}) +  Mu(x_{0}) > Mu(x_{0})
\end{split}
\]
\\
The last inequality follows from the strict positivity of the function $\varphi$. But we observe that the inequality contradicts the obstacle constraint $u(x_{0}) \leq Mu(x_{0}).$ Hence we have reached our desired contradiction.
\\
\\
Finally the last statement of the claim follows from the continuity of $u$. The continuity of the solution implies that $\{u < Mu \}$ is an open set and thus in a small neighborhood $N_{1}$ of $x$, $u$ satisfies the equation, $F(D^{2}u) = 0.$ We can therefore apply interior regularity estimates to conclude. 
\end{proof}

We now strenghten the previous claim to obtain a uniform neighborhood of $\Sigma_{x_{0}}$ that is strictly contained in the non-contact region.
\begin{claim} $\exists \delta_{0} >0$ such that $d \left \{\{u = Mu\}, \Sigma_{x_{0}}\right \} > \delta_{0}$.
\end{claim}

\begin{proof}
Suppose by contradiction $\exists \{\delta_{k}\} \searrow 0$ and $\{x_{k}\} \subset \Sigma_{x_{0}}$, such that 
$$d(x_{k}, \{u = Mu\}) < \delta_{k}.$$
By definition, $x_{k} \in \Sigma_{x_{0}}$, implies
$$\varphi(x_{0}) + u(x_{k}) = Mu(x_{0}) \; \; \forall k.$$ 
By the continuity of $u(x)$ this implies in particular that $\varphi(x_{0}) + u(\bar{x}) = Mu(x_{0})$ for some $\bar{x} \in \{u = Mu\}$. On the other hand, $\varphi(x_{0}) + u(\bar{x}) = Mu(x_{0})$ implies $\bar{x} \in \Sigma_{x_{0}}$. Hence from the previous claim we obtain,
$$Mu(\bar{x}) = u(\bar{x}) < Mu(\bar{x}).$$
This is our desired contradiction.
\end{proof}

We now state and prove a claim which allows us to redefine the obstacle in the neighborhood of a contact point.
 
\begin{claim} For every $x, \bar{x} \in \Omega$, $\exists \delta >0$, such that if $|x -x_{0}| < \delta$, and $d(\bar{x}, \Sigma_{x_{0}}) >\delta$, then $u(x) < \varphi(x) + u(\bar{x})$. Moreover, if $x \in \{u = Mu \}$, then $\bar{x} \notin \Sigma_{x}$. 
\end{claim}

\begin{proof}
Suppose by contradiction that there exists a sequence of points $\{x_{k}\}$ and $\{\bar{x}_{k'}\}$ satisfying:
\\
1. $|x_{k}-x_{0}| = \delta_{k}.$
\\
2. $d(\bar{x}_{k'},\Sigma_{0}) > \delta_{k'} > 0$.
\\
3. $\{\delta_{k}\} \searrow 0$ and $\{\delta_{k'}\} \searrow 0$.
\\
4. $u(x_{k}) \geq \varphi(x_{k}) + u(\bar{x}_{k'}) \;$ $\forall k$ and $\forall k'$.
\\
\\
We observe that from the previous claim $\exists k_{0}, k_{0}'$, such that $\forall k \geq k_{0}$ we have the following chain of inequalities,
\\
\[
\begin{split}
 Mu(x_{0} + \delta_{k})  & \leq Mu(\bar{x}_{k_{0}'}) \\
 &  \leq \varphi(\bar{x}_{k_{0}'}) + u(\bar{x}_{k_{0}'}) \\
 & \leq \varphi(x_{0} + \delta_{k}) + u(\bar{x}_{k_{0}'}) \\
 & \leq u(x_{0} + \delta_{k}) \leq Mu(x_{0} + \delta_{k}).
\end{split}
\]
\\
Thus the above inequalities are all equalities. This implies $\forall k \geq k_{0}$, 
$$Mu(x_{0} + \delta_{k}) = \varphi(x_{0} + \delta_{k}) + u(\bar{x}_{k_{0}'}).$$ 
Letting $k \to \infty$ we obtain, 
$$ Mu(x_{0}) = \varphi(x_{0}) + u(\bar{x}_{k_{0}'}).$$
Which implies in particular that $\bar{x}_{k_{0}'} \in \Sigma_{x_{0}}$. This is our desired contradiction. 
\end{proof}

From the last claim we can redefine the obstacle for $V_{\delta} = \{ |x-x_{0}| < \delta\}$. In particular by taking $\delta$ sufficiently small $\exists N_{2}$ neighborhood of $\Sigma_{x_{0}}$ such that,
$$Mu(x) =  \varphi(x) + \inf_{\stackrel{\xi \geq 0}{x + \xi \in N_{2}}}u(x + \xi).$$ 
For an even smaller $\delta$, 
$$Mu(x) =  \varphi(x) + \inf_{\stackrel{\xi \geq 0}{x_{0} + \xi \in N_{3}}}u(x + \xi).$$ 
Where  $N_{3}$ is such that, 
$$V_{\delta} + N_{3} - x_{0} \subseteq N_{1}.$$   
Here $N_{1}$ is the neighborhood obtained in Claim 3. In particular for $x \in V_{\delta}$ and $\xi \in N_{3} - x_{0}$, we can bound the second incremental quotients.  
$$\delta^{2}u = u(x + h + \xi) + u(x - h + \xi) - 2u(x + \xi) \leq c|h|^{2}.$$  
Moreover we know that for some $x + \bar{\xi}$ in $N_{1}$, we have, 
$$ \inf_{\stackrel{\xi \geq 0}{x + \xi \in N_{1}}}u(x + \xi) = u(x + \bar{\xi}).$$ 
Now we consider the second incremental quotients of the obstacle $Mu(x).$ By the semi-concavity of $\varphi$ we obtain,
\\
\[
\begin{split}
 \delta^{2}Mu(x) & \leq \omega(h) + u(x + \bar{\xi} + h) + u(x + \bar{\xi}  - h) - 2u(x + \bar{\xi})\\
 & \leq \omega(h) + c|h|^{2}  \\
 &\leq C \omega(h).
\end{split}
\]
\\
Thus, in a neighborhood of a contact point, $Mu(x)$ is semi-concave with semi-concavity modulus $ \omega(h)$.
\\
\\
\textbf{Case 2}: Fix $x \in \{u < Mu\}$. We argue as before by considering the second incremental quotients of the obstacle, $\delta^{2}Mu(x)$. We observe that the infimum of $u$ in the positive cone, $\xi \geq 0$, must always be realized at a non-contact point. Suppose $ \exists \; x + \xi_{1} \in \{u = Mu \}$ satisfing,
$$ \inf_{\stackrel{\xi \geq 0}{x + \xi \in \Omega}}u(x + \xi) = u(x + \xi_{1}).$$
Then from \textbf{Case 1} there exists $\xi_{2} \in \Sigma_{x + \xi_{1}} \subset \{u < Mu \}$ such that,  
$$\inf_{\stackrel{\xi \geq 0}{x + \xi_{1} + \xi \in \Omega}}u(x + \xi_{1} + \xi) = u(x + \xi_{1} + \xi_{2}).$$
Since $\xi_{1} + \xi_{2} \geq 0$, we have found a positive vector admissiable to 
$$ \inf_{\stackrel{\xi \geq 0}{x + \xi \in \Omega}}u(x + \xi).$$ 
Furthermore, $u(x + \xi_{1} + \xi_{2}) \leq u(x + \xi_{1})$. Hence we conclude that for a fixed $x \in \{u < Mu \}$ and for some $x + \bar{\xi}$ in $\{u < Mu\}$,
$$ \inf_{\stackrel{\xi \geq 0}{x + \xi \in \Omega}}u(x + \xi) = u(x + \bar{\xi}).$$ 
Moreover from Claim 4 we know that $x + \bar{\xi}$ is a uniform positive distance away from the contact set $\{u = Mu \}$. Hence there exists a uniform neighborhood $N_{0}$ of points around $x + \bar{\xi}$ where $\{u < Mu\}$. In a smaller neighborhood $N_{1}$, $u \in C^{1, 1}(N_{1})$. In particular for $x + \xi \in N_{1}$, we can bound again the second incremental quotients,  
$$u(x + h + \xi) + u(x - h + \xi) - 2u(x + \xi) \leq c|h|^{2}$$ 
Using once more the semi-concavity estimate on $\varphi(x)$ and for some $x + \bar{\xi}$ in $N_{1}$ we find,
\\
\[
\begin{split}
 \delta^{2}Mu(x)  & \leq \omega(h) + u(x + \bar{\xi} + h) + u(x + \bar{\xi}  - h) - 2u(x + \bar{\xi})\\
 & \leq \omega(h) + c|h|^{2}  \\
 &\leq C \omega(h).
\end{split}
\]
\\
Thus, in a neighborhood of a non-contact point, $Mu(x)$ is semi-concave with semi-concavity modulus $\omega(h)$.
\end{proof}

We are now in position to apply the general estimate obtained in the previous section.
\begin{theorem}
Let u be the solution to the fully nonlinear stochastic impulse control problem
\\
\[
   \begin{cases}
        F(D^{2}u) \geq 0 & \forall x \in \Omega. \\
        u(x) \leq \textnormal{M}u(x) & \forall x \in \Omega.\\
        u = 0 & \forall x \in \partial \Omega.       
  \end{cases}
\]
\\
Let $\Omega \subset \mathbb{R}^{n}$ a bounded domain with a $C^{2,\alpha}$ boundary $\partial \Omega$, and  define the obstacle,
$$Mu(x) = \varphi(x) + \inf_{\stackrel{\xi \geq 0}{x + \xi \in \bar{\Omega}}}(u(x + \xi)).$$
Where $\varphi(x)$ is $\omega(r)$ semi-concave, strictly positive, bounded, and decreasing in the positive cone $\xi \geq 0$. Then, the solution $u$ has a modulus of continuity $\omega(r)$ up to $C^{1,1}(\Omega)$.
\end{theorem}

\begin{proof}
The previous theorem shows that $Mu(x)$ is $\omega(r)$ semi-concave. Set $\varphi_{u}(x) = -Mu(x)$ and note that $\varphi_{u}$ is $\omega(r)$ semi-convex. We apply the estimates from the previous section to conclude.
\end{proof}

Finally as in the classical case, assuming analytic data and $f(x) \leq f(x + \xi) \; \; \forall \xi \geq 0$, as well as concavity of $F(\cdot)$ in the hessian variable, it follows from an application of a nonlinear version of the Hopf Boundary Point Lemma \cite{BDl99} and the results of \cite{Lee98} that we obtain the following structural theorem for the free boundary, 
\begin{theorem}
Given the Fully Nonlinear Stochastic Impulse Control Problem
\begin{equation}
   \begin{cases}
        F(D^{2}u) \geq f & \forall x \in \Omega.\\
        u(x) \leq \textnormal{M}u(x) = 1 + \inf_{\stackrel{\xi \geq 0}{x + \xi \in \Omega}}u(x + \xi). & \forall x \in \Omega.\\
        u = 0 & \forall x \in \partial \Omega.        
  \end{cases}
\end{equation}

It follows that, $\partial \{u < \textnormal{M}u \} = \Gamma^{1}(u) \cup \Gamma^{2}(u)$ where,
\\
1. $\forall x_{0} \in \Gamma^{1}(u)$ satisfying a uniform thickness condition on the coincidence set $\{u = \textnormal{M}u \}$,  there exists some appropriate system of coordinates in which the coincidence set is a subgraph $\{x_{n} \leq g(x_{1}, \dots, x_{n-1}) \}$ in a neighborhood of $x_{0}$ and the function $g$ is analytic.
\\
2. $\Gamma^{2}(u) \subset \Sigma(u)$ where $\Sigma(u)$ is a finite collection of $C^{\infty}$ submanifolds.
\end{theorem}

\begin{remark} We point out that the above theorem holds for the more general implicit constraint obstacle 
$$\textnormal{M}u = h(x) + \inf_{\stackrel{\xi \geq 0}{x + \xi \in \Omega}}u(x + \xi)$$
where the regularity of $\Gamma^{1}(u)$ corresponds to the regularity of $h(x)$.
\end{remark}

\section{Applications to a Penalized Problem}

In this section we study a penalized fully nonlinear obstacle problem. The goal is to obtain optimal uniform estimates in the penalizing paramter $\epsilon.$ For this section we fix the modulus of semiconvexity to be linear, i.e. $\omega(r) = cr^{2}$. We point out that the followng can be suitably modified for a general modulus of semiconvexity. The idea to obtain the optimal estimate is to use the interplay between semiconvexity of the obstacle and the superharmonicity of the equation as before.

\begin{lemma}  Consider the fully nonlinear penalized obstacle problem with obstacle $\varphi_{u}$, admitting a modulus of semiconvexity, $\omega(r) = Cr^{2}$ and a suitably defined class of penalizations $\beta_{\epsilon}$,
\begin{equation}
   \begin{cases}
        F(D^{2}u) = \beta_{\epsilon}(u-\varphi_{u}) & \Omega,\\
        u = 0 & \partial \Omega,\\
       \varphi_{u} < 0 & \partial \Omega.      
  \end{cases}
\end{equation}
Then the solution $u$ has a modulus of continuity $\omega(r)$ up to $C^{1,\alpha}(\Omega) \; \forall \alpha < 1$ independent of the penalizing parameter $\epsilon$.
\end{lemma}

\begin{proof}
Let $\rho(x)$ be a function in $C^{\infty}(\mathbb{R}^{n})$ with support in the unit ball, such that $\rho \geq 0$ and $\int_{\mathbb{R}^{n}} \rho = 1.$ Define for any $\delta > 0$,
$$\rho_{\delta}(x) = \delta^{-n}\rho(\frac{x}{\delta}).$$

Consider the mollifier
$$J_{\delta} [\varphi_{u}](x) = \int_{\Omega} \rho_{\delta}(x-y) \varphi_{u}(y) \; dy.$$

Recall that $\varphi_{u}$ semi-convex with a linear modulus implies that for any $\xi \in C^{\infty}_{0}(\Omega_{0})$, $\xi \geq 0$, where $\Omega_{0} \subset \Omega$ is an open set, it holds that for any directional derivative, $\frac{\partial}{\partial \eta}$ and some constant $C > 0$ independent of $\delta$,
$$\int_{\Omega} \varphi_{u} \frac{\partial^{2} \xi}{\partial \eta^{2}} \geq -C.$$

Taking $\xi = \rho_{\delta}$, it follows that pointwise in $\Omega$,
$$\frac{\partial^{2} J_{\delta} [\varphi_{u}]}{\partial \eta^{2}} \geq -C.$$

We consider, $\varphi_{u}^{\delta} = J_{\delta} [\varphi_{u} + \frac{C}{2} |x|^{2}] -  \frac{C}{2} |x|^{2}.$ It follows that
$$|D \varphi_{u}^{\delta}| \leq C.$$
$$\frac{\partial^{2} \varphi_{u}^{\delta}}{\partial \eta^{2}} \geq -C.$$
$$\varphi_{u}^{\delta} \to \varphi_{u} \; \; \textnormal{uniformly in} \; \Omega \; \textnormal{as} \; \delta \to 0.$$

Define $\beta_{\epsilon}(t) \in C^{\infty}$ for $0 < \epsilon < 1$ and $C$ a constant independent of $\epsilon$, such that,
\\
1. $\beta'_{\epsilon}(t) > 0.$ 
\\
2. $\beta_{\epsilon}(t) \to 0 \; \; \textnormal{if} \; t > 0, \epsilon \to 0$.
\\
3. $\beta_{\epsilon}(t) \to -\infty \; \; \textnormal{if} \; t < 0, \epsilon \to 0$.
\\
4. $\beta_{\epsilon}(t) \leq C$
\\
5. $\beta_{\epsilon}''(t) \leq 0$.
\\
Consider the penalized problem,
\begin{equation}
   \begin{cases}
        F(D^{2}u) - \beta_{\epsilon}(u-\varphi_{u}^{\epsilon}) = 0 & \Omega,\\
        u = 0 & \partial \Omega.      
  \end{cases}
\end{equation}

Define for $N  > 0,$
$$\beta_{\epsilon , N}(t) = \max \{ \min \{\beta_{\epsilon}, N \}, -N \}.$$

Consider the problem,
\begin{equation}
   \begin{cases}
        F(D^{2}u) - \beta_{\epsilon, N}(u-\varphi_{u}^{\epsilon}) = 0 & \Omega,\\
        u = 0 & \partial \Omega.      
  \end{cases}
\end{equation}

It follows from $W^{2,p}$ theory for fully nonlinear equations that for each $v \in L^{p}(\Omega) \cap C^{0}(\bar{\Omega})$ $(1 < p < \infty)$, there exists a unique solution $w \in W^{2,p}(\Omega) \cap C^{0}(\bar{\Omega})$ solving,
\begin{equation}
   \begin{cases}
        F(D^{2}w) - \beta_{\epsilon, N}(v-\varphi_{u}^{\epsilon}) = 0 & \Omega,\\
        u = 0 & \partial \Omega,      
  \end{cases}
\end{equation}
and for $\bar{C}$ independent of $v$,
$$\|w\|_{W^{2,p}} \leq \bar{C}.$$

Define the solution map $T$ such that $Tv = w.$ Notice that $T$ maps $B_{\bar{C}}(0) \subset L^{p}(\Omega)$ into itself and is compact. Hence by Schauder's fixed-point theorem, it follows, that there exists $u$ such that $Tu = u$. In particular, we have found a solution to $(16)$. Moreover $\beta_{\epsilon, N}(u-\varphi_{u}^{\epsilon}) \in C^{0,\alpha}$. Hence by Evans-Krylov $\|u\|_{C^{2,\alpha}} \leq C(\epsilon).$ We now estimate $\zeta = \beta_{\epsilon, N}(u-\varphi_{u}^{\epsilon}).$ By definition we know that $\beta_{\epsilon, N}(u-\varphi_{u}^{\epsilon}) \leq C$ for a constant $C$ independent of $N, \epsilon.$ Let $x_{0}$ be the minimum point of $\zeta.$ Without loss of generality we assume,
$$\mu = \zeta(x_{0}), \; \; \mu \leq 0, \; \; \mu < \beta_{\epsilon}(0).$$

It follows that $x_{0} \notin \partial \Omega.$ If not, then,
$$\mu = \zeta(x_{0}) = \beta_{\epsilon, N}(-\varphi_{u}^{\epsilon}) \geq \beta_{\epsilon, N}(0) \geq \beta_{\epsilon}(0).$$
A contradiction. On the other hand if $x_{0} \in \Omega$, then $\beta_{\epsilon}'(t) \geq 0$ implies that,  $$\min_{\Omega}(u-\varphi_{u}^{\epsilon}) = u-\varphi_{u}^{\epsilon}(x_{0}) < 0.$$

Moreover it follows that $D^{2}(u-\varphi_{u}^{\epsilon})(x_{0}) \geq 0$. Hence $F(D^{2}(u-\varphi_{u}^{\epsilon})(x_{0})) \geq 0.$ By Ellipticity and the semiconvexity estimate it follows that,
\[
\begin{split}
\beta_{\epsilon, N}(u-\varphi_{u}^{\epsilon}) (x_{0}) & = F(D^{2}u_{\epsilon, N} - D^{2} \varphi_{u}^{\epsilon} + D^{2} \varphi_{u}^{\epsilon}) \\
             & \geq F(D^{2}u_{\epsilon, N} - D^{2} \varphi_{u}^{\epsilon}) +\lambda \| D^{2} (\varphi_{u}^{\epsilon})^{+} \| - \Lambda  \| D^{2} (\varphi_{u}^{\epsilon})^{-} \| \\
             & \geq -C. \\
\end{split}
\]

In particular, $|\beta_{\epsilon, N}(u-\varphi_{u}^{\epsilon}) | \leq C$ for a constant $C$ independent of $\epsilon$ and $N$. Furthermore $|F(D^{2}u)| \leq C$. It follows from elliptic estimates,
$$\|u\|_{W^{2,p}} \leq C.$$
Hence for $N$ large enough $u$ is a solution for the penalized problem (15). 
\end{proof}

We now prove the optimal estimate as before

\begin{theorem}  Consider the solution to the fully nonlinear penalized obstacle problem with obstacle $\varphi_{u}$, admitting a modulus of semiconvexity, $\omega(r) = Cr^{2}$ and a suitably defined class of penalizations $\beta_{\epsilon}$. Assume that $F(D^{2}u)$ is convex in the Hessian variable. Then the solution $u$ is $C^{1,1}$ independent of $\epsilon.$
\end{theorem}

\begin{proof}
Consider the penalization problem
\begin{equation}
   \begin{cases}
        F(D^{2}u) = \beta_{\epsilon}(u-\varphi_{u}) & \Omega,\\
        u = 0 & \partial \Omega.\\
        \varphi_{u} < 0 & \partial \Omega.        
  \end{cases}
\end{equation}
We aim to bound $\inf u_{\tau \tau}$ from below.  The following computation continues to hold for viscosity solutions by using incremental quotients and recalling that second order incremental quotients are supersolutions of a convex equation. We fix a directional derivative $\tau$ and differentiate the penalization identity to obtain,
$$F_{ij,kl}(D^{2}u)(D_{ij}u_{\tau})(D_{kl}u_{\tau}) + F_{ij}(D^{2}u)(D_{ij}u_{\tau \tau}) =$$ 
$$\beta_{\epsilon}''(u - \varphi_{u})(u - \varphi_{u})_{\tau}^{2} + \beta_{\epsilon}'(u - \varphi_{u})(u - \varphi_{u})_{\tau \tau}.$$

By convexity of the operator and the structural conditions on the penalization family $\beta_{\epsilon}(t)$ it follows that
$$F_{ij}(D^{2}u))(D_{ij}u_{\tau \tau}) \leq \beta_{\epsilon}'(u - \varphi_{u})(u - \varphi_{u})_{\tau \tau}.$$

Suppose the minimum point of $u_{\tau \tau}$ is in the interior of the domain then, since $\beta'(t) > 0$, we find $(u - \varphi_{u})_{\tau \tau} \geq 0$. In particular, $u_{\tau \tau} \geq -C.$ Suppose now that the minimum point of $u_{\tau \tau}$ is realized on the boundary of the domain. We differentiate the equation with respect to $x_{\tau}$ for $\tau \in \{1, \dots, n-1\}$ and obtain,
$$F_{ij}(D^{2}u)D_{ij}u_{\tau} = \beta_{\epsilon}'(u - \varphi_{u})(u - \varphi_{u})_{\tau}.$$
Recall $\varphi_{u} < 0$ on $\partial \Omega$. Hence for a fixed $\epsilon_{0} > 0$ it follows that $\varphi_{u} \leq u + \epsilon_{0}$ in $\{x \in \bar{\Omega} \; | \; d(x, \partial \Omega) \leq \frac{\epsilon_{0}}{2} \}.$ Morevover by the uniform continuity of $u^{\epsilon} \to u$ on $\bar{\Omega}$, there exists a small $\epsilon_{1}$, such that $\varphi_{u} \leq u + \frac{\epsilon_{0}}{2}$, $|\beta'| < \epsilon_{0}$, and $|\beta''| < \epsilon_{0}$ in $\{x \in \bar{\Omega} \; | \; d(x, \partial \Omega) \leq \frac{\epsilon_{0}}{2} \}$ for $0 < \epsilon < \epsilon_{1}.$ Hence it follows from the boundary H\"{o}lder estimates for linear non-divergence form equations,
$$\|u_{\tau n} \|_{L^{\infty}(\partial B^{+}( \frac{\epsilon_{0}}{4}))} \leq C.$$

Moreover by uniform ellipticity we can use the equation to solve for $u_{nn}$ in terms of $\beta$ and $u_{kl}$ for $k \in (1, \dots, n-1)$ and $l \in (1, \dots , n).$ Hence we obtain after straightening the boundary,
$$\|D^{2} u\|_{L^{\infty}(\partial \Omega)} \leq C.$$ 

Hence it follows that the solution is semiconvex with a linear modulus. Moreover $F(D^{2}u) = \beta_{\epsilon}(u - \varphi_{u}) \leq 0$ after choosing a penalization satisfying $\beta_{\epsilon}(t) \leq 0$ for $t \geq 0$. Finally an application of Lemma 6 proves that $u$ has a uniform $C^{1,1}$ estimate. 
\end{proof}

\begin{remark} We point out that the above arguments give us a straightforward proof for $C^{1,1}$ estimates when the operator is convex. The previous section was based on $C^{1,\alpha}$ estimates for Fully Nonlinear equations hence did not have a restriction on the sign of the operator.
\end{remark}

\begin{remark} Previous computation and estimates can be generalized to Viscosity Solutions of convex operators (see \cite{CC95}). 
\end{remark}

Finally, as an application of the uniform estimates, we prove how the $C^{2,\alpha}$ estimate for the penalized problem decays in the penalizing parameter.

\begin{corollary} Consider the fully nonlinear penalized obstacle problem with obstacle $\varphi_{u}$, admitting a modulus of semi-convexity, $\omega(r) = Cr^{2}$ and a suitably defined class of penalizations $\beta_{\epsilon}$,
\begin{equation}
   \begin{cases}
        F(D^{2}u) = \beta_{\epsilon}(u-\varphi_{u}) & \Omega,\\
        u = 0 & \partial \Omega.\\
        \varphi_{u} < 0 & \partial \Omega.        
  \end{cases}
\end{equation}
Moreover assume that $F( \cdot)$ is convex in the hessian variable. Then for a constant $C$ independent of $\epsilon$,
$$\|u\|_{C^{2,\alpha}} \leq C\epsilon^{-\alpha}.$$
\end{corollary}

\begin{proof}
It is well known that the penalization problem converges to the obstacle problem independent of the choice of penalizing family. Hence we fix a penalizing family, 
\begin{equation}
 \beta_{\epsilon}(t) = \left\{
     \begin{array}{lr}
       \frac{t}{\epsilon^{2}} &  \; \; t < 0.\\
       0 & t \geq 0.
     \end{array}
   \right.
\end{equation}

Without loss of generality we fix $\varphi_{u} = 0$. The following argument can be modified for non-zero obstacle $\varphi_{u}$. We consider the scaled function,
$$v^{\epsilon}(x) = \frac{1}{\epsilon^{2}}u^{\epsilon}(\epsilon x).$$ 
We note that
$$F(D^{2}v^{\epsilon}) = F(D^{2} (\frac{1}{\epsilon^{2}}u^{\epsilon}(\epsilon x))) = F(D^{2} u^{\epsilon}(\epsilon x)) = \frac{1}{\epsilon^{2}}u^{\epsilon}(\epsilon x) = v^{\epsilon}(x).$$ 

Hence we obtain for a constant $C$ independent of $\epsilon$,
$$\|v^{\epsilon}\|_{C^{2, \alpha}} \leq C.$$

It follows,
\[
\begin{split}
|D^{2}u^{\epsilon}(x) - D^{2}u^{\epsilon}(y)| & = |\epsilon^{2}D^{2}v^{\epsilon}(\frac{x}{\epsilon}) - \epsilon^{2}D^{2}v^{\epsilon}(\frac{y}{\epsilon})|  \\
             & = |D^{2}v^{\epsilon}(\frac{x}{\epsilon}) - D^{2}v^{\epsilon}(\frac{y}{\epsilon})| \\
             & \leq C | \frac{x}{\epsilon} - \frac{y}{\epsilon}|^{\alpha} \\
             & \leq C\epsilon^{-\alpha} |x-y|^{\alpha}.
\end{split}
\]
\end{proof}

\end{document}